\documentclass[preprint,12pt,3p]{elsarticle}
\usepackage{graphicx}
\usepackage[utf8]{inputenc}
\usepackage[T1]{fontenc}
\usepackage{textcomp}
\usepackage{hyperref}
\usepackage{amsmath}
\usepackage{mathtools}
\hypersetup{colorlinks,linkcolor={blue},citecolor={blue},urlcolor={red}}

\usepackage{amsmath}
\usepackage{amssymb}
\usepackage{amsthm}
\newtheorem{theorem}{Theorem}[section]
\newtheorem{lemma}[theorem]{Lemma}

\newtheorem{definition}[theorem]{Definition}

\newtheorem{example}[theorem]{Example}

\newtheorem{remark}[theorem]{Remark}

\numberwithin{equation}{section}

\usepackage{amssymb}

\begin{document}

\begin{frontmatter}

\title{Frational p-Laplacian on Compact Riemannian Manifold}

\author[label1]{A. Ouaziz}
\address[label1]{ Faculty of Sciences Dhar Elmahraz, Sidi Mohamed Ben Abdellah University,  Fez, Morocco.}
\ead{abdesslamaitouaziz@gmail.com}

\author[label2]{A. Aberqi}
\address[label2]{Sidi Mohamed Ben Abdellah University, National School of Applied Sciences Fez, Morocco.}
\ead{ahmed.aberqi@usmba.ac.ma}
\begin{abstract}
In this paper, we investigate  the existence and uniqueness  of a non-trivial solution for a class of nonlocal equations involving the fractional $p$-Laplacian operator defined on  compact Riemannian manifold, namely,
\begin{eqnarray}\label{k1}
\begin{gathered}
\left\{\begin{array}{lll}
 (-\Delta_g)^s_p u(x)+ \left| u \right|^{p-2} u= f(x,u) & \text { in }& \Omega,\\
 
 \hspace{3,4cm} u=0 & \text{in }& M\setminus\Omega,
 \end{array}\right.
\end{gathered}
\end{eqnarray}
 and  $\Omega$ is an open bounded subset of M with a smooth boundary.
\end{abstract}
\end{frontmatter}
\section{Introduction}
Non-local operator problems have recently received a lot of attention in the literature. A good amount of investigation have focused on the existence and regularity of solutions to such problems governed by the fractional Laplacian in Euclidean space $\mathbb{R}^n$, see for instance \cite{  CG, MG, GQ, r10, r12,   M, r-4, r17}. For the prototype form $(-\Delta)^s$, $0<s<1$, which is the infinitesimal generator of the $s$-stable L\'evy processes \cite{r2}, the associate equation was treated by R. Servadei  and E. Valdinoci in  reference \cite {r17} by proving the existence of a  solution to  the following problem:
\begin{eqnarray}\label{k1}
\begin{gathered}
\left\{\begin{array}{lll}
 (-\Delta)^s u(x) = f(x,u) & \text{in }& \Omega,\\
 \hspace{1,45cm}u=0 & \text{in }& \mathbb{R}^{N}\setminus \Omega,
\end{array}\right.
\end{gathered}
\end{eqnarray}
where $(-\Delta)^s$ is a non- local operator defined as follow: 
$$ (-\Delta)^s (u(x))=  C(N, s) PV \int_{{R}^{N}} \frac{u(x)-u(y)}
{\left| {x-y} \right|^{N+ps}}dy,$$
see for instante the reference \cite{r17} for more details.
For the non-linear involving the $p$-fractional Laplacian which is defined by
\begin{equation*}
(-\Delta)^s_p u(x)=2\lim_{\varepsilon\to 0^{+}}
\int_{{R}^{N} \backslash B_\varepsilon(x)}\frac{\left| {u(x)-u(y)} \right|^{p-2}(u(x)-u(y))}
{\left| {x-y} \right|^{N+ps}}dy,
\end{equation*}
 where $ sp<N $ with $ s\in(0,1)$, $p\in (1, \infty), \Omega $ is an open bounded subset of $\mathbb{R^{N}}$ with smooth boundary  and  $ B_\varepsilon(x)$ the ball of  $ \mathbb{R}^{N}$ of center $x$ and radius $\epsilon$, we refer to \cite {CG,  MG, GQ, r10, r12, M, r-4}. 
This type of problem arises in many applications such as continuum mechanics, phase phenomena \cite{A-ph}, population dynamics,  game theory , crystal dislocation \cite{dislocation}, optimization, finance \cite{conte, duvaut}, stratified materials,  conversation laws  and minimals surfaces \cite{bile, caffarelli, valdinici, savin}.\\
For the framework of the above problems on the open set of $N$-dimensional, real Euclidean space $\mathbb{R}^N$, we recommend  \cite{Radulescu, r8}.\\
Here, in this paper, we are interested in the non-Euclidian case, i.e., the case of Riemannian manifold, and the treatment of non-linear fractional operators defined on Riemannian manifold defined as follow:
\begin{equation*}
(-\Delta_g)^s_p u(x)=2\lim_{\varepsilon\to 0^{+}}
\int_{M\backslash B_\varepsilon(x)}\frac{\left| {u(x)-u(y)} \right|^{p-2}(u(x)-u(y))}
{(d_g(x,y))^{N+ps}}\,d\mu_g(y),
\end{equation*}
for $x\in M$,  where  $ M$ is a compact Riemannian $N-$manifold,  $ d\mu_g(y)= \sqrt{\text{det}(g_{ij})}dy$ is the Riemannian volume element on $(M,g),   B_\varepsilon(x)$ denotes the geodessic ball of centre $x$  and radius $\epsilon,  dy $ is the Lebesgue volume element of $\mathbb{R^{N}}, d_{g}(x, y) $ defines a distance on $M,$ and $g$ a $\mathcal{C}^\infty$ Riemannian metric, see  the reference \cite{aubin} and section 2 for more information. Consider the following non-linear problem with lower order:
\begin{eqnarray}\label{k1}
\begin{gathered}
\left\{\begin{array}{lll}
 (-\Delta_g)^s_p u(x) +  \left| u \right|^{p-2} u= f(x,u) & \text{in } &\Omega,\\
\hspace{3,15cm} u=0  &\text{in }& M\setminus\Omega,
\end{array}\right.
\end{gathered}
\end{eqnarray}

 where $N> ps$ with $s\in(0,1)$, $p\in (1, \infty)$, and  $\Omega$ is an open bounded subset of $M$ with
 smooth boundary $\partial\Omega$, $f: \Omega \times \mathbb{R}\to\mathbb{R}$ is a Carath\'{e}odory
function satisfying the following conditions:
\begin{itemize}
 \item[(f1)]
 There exist $ \beta>0 $ and $ 1<q<p_s^*=\frac{Np}{N-ps}$ such that
$$
\left |f(x,t) \right| \leq \beta (1+\left| t\right| ^{q-1}),
$$
for a.e.\ $x\in\Omega$,\ $ t \in\mathbb{R}$.

 \item[(f2)] For $ 1<q<p_s^*$, we have
$$ \lim_{t\to \infty}\frac{f(x,t)}{\left |t \right|^{q-1}}=0  \,\,\,     \text {uniformly for a.e}
  \,\,\, x \in \Omega. $$ 
 
\item[(f3)] 
$$\lim_{\zeta\to 0} \frac{f(x,\zeta)}{\left|\zeta\right |^{p-1}} = 0\,\,\, \text {uniformly for a.e}\,\,\, x \in \Omega;$$
\item[(f4)]
(AR condition) There exists $\mu> p$ such that $$0<\mu F(x,t)\leq t f(x,t)  \text{  for a.e }   x\in \Omega  \text{  and  }\, t>0,$$
where  $\displaystyle  F(x,t)= \int_0^t f(x,s)ds$.
 \item[(f5)]
The function  $ h:\Omega \times  (0,\ \infty)  \to \mathbb{R^{+}}$ as
   \[
\ h(x,t)=\frac{f(x,t)}{t^{p-1}}
\] is decreasing in $(0, \infty)$ for a.e $ x \in \Omega$.

\end{itemize}
\begin{example}
 For $c>0$, the function $f(x,t)= c \left|t\right |^{p} \exp(-t)$  satisfying the above conditions.
\end{example}
\begin{remark}
The conditions {\rm (f1)}--{\rm (f4)} are used to prove the solution's existence, with the auxiliary condition {\rm (f5)} proving the solution's uniqueness.
\end{remark}
Our goal in this paper is to prove the existence and uniqueness of a non-trivial solution, via variational methods in the framework of fractional Sobolev space on Riemannian manifold, in which, we extend the results proved in  references \cite{aubin, CG, r10,  druet, guo2,  ilyas, trudinger} in the  non-Euclidean case, this generates some complications due to the non locality character of the operator, that is having a well defined Dirichlet problem in the non-local framework. So it  is not enough to prescribe the boundary condition $\partial \Omega$, since to compute the value of $(-\Delta_g)^s_p u(x)$ at $x\in \Omega$, we need to know the value of $u(x)$ in the whole $M$. Other complications are due to the non-Euclidean framework of our equation. For that, checking for example the density of Shwartz space $\mathcal{D}(M)$ in $W^{s,p}(M)$, it's not useful to consider a function $\mathcal{C}^\infty$ on $\mathbb{R}$, as in the proof of Euclidean case, because for Riemannian manifold $[d(P,Q)]^2$ is only Lipschitz function in $M$ and in $Q\in M$, $P$ being fixed point of $M$. For more functional properties of Sobolev space to compact Riemannian manifold, we refer to \cite{aubin, FSV1}. In addition, another challenge is to verify that the chosen test functions are admissible. \\

 The rest of this paper is structured as follows: In section \ref{sec2} we recall some Definitions, Lemmas, and Theorems, that will help us in our analysis. In section \ref{sec4} we will prove the existence of a non-trivial solution using the Mountain Pass Theorem, and in section \ref{sec5}, we will show the uniqueness of a non-trivial solution of our problem.

\section{  Background Material}\label{sec2}
First of all, we recall the most important and relevant properties and notations, by  referring to \cite{aubin, benslimane2020existence, FSV1} for more details.
\begin{definition} 
Let $(M,g)$ be an $N$-dimensional Riemannian manifold
and let $\nabla$ be the Levi-Civita connection.
For $u\in C^\infty(M)$, then $\nabla^k u$ denotes the $k-$th covariant
derivative of $u$. In local coordinates, the pointwise norm of $\nabla^k u$
is given by $$
 \left |\nabla^{k}u \right| = g^{i_1j_1}\cdot\cdot\cdot g^{i_kj_k} (\nabla^k u)_{i_1i_2\dots i_k} (\nabla^k u)_{j_1j_2\dots j_k}.$$

When $k=1$, the components of $\nabla u$  in local coordinates are given by
$$(\nabla u)_i=\nabla^iu$$. By definition, one has that
 $$
   \left |\nabla u \right| = \sum_{i,j=1}^\infty g^{ij}\nabla^iu\nabla^ju.
$$
\end{definition} 
 \begin{definition}
 Let $(M,g)$ be an $N$-dimensional Riemannian manifold
and \\
$\gamma:[a,b]\subset\ \mathbb {R}  \longrightarrow  M $ a curve of class $ C^{1}$,  the length of $c$  is  $$ l(\gamma)= \int_a^b (g(\gamma'(t), \gamma'(t)))^{\frac{1}{2}} dt. $$
\end{definition}

 \begin{definition}
 
Let $(M,g)$ be an $N-$dimensional Riemannian manifold, and let $C^{1} _{x,y}$  be the space of piecewise  $ C^{1}$ curves $ \gamma:[a,b]\subset\ \mathbb {R}  \longrightarrow  M $ such that $\gamma(a)=x$  and $\gamma(b)=y$. The distance between $x$ and $y$ is defined by $$ d(x,y)=   \mbox {inf}\{  l(\gamma),     \text{ $\gamma$ is a differentiable curve connecting $x$ and $y$  } \}.$$
\end{definition}

 \begin{remark}
  Let $(M,g)$ be an $N$-dimensional Riemannian manifold, then $(M,g)$  is a  metric space.
 \end{remark}
 \begin{definition}
 Consider $X$ to be a topological space, and $ U=\{ U_{i},  i\in I\}$  to be an open covering of X. A subordinate to $ U $ is a family of finite functions that satisfy two conditions, $ \sum_{i \in I} \eta  _{i}=1$ and $ Supp\
  \eta  _{i} \subset  U_{i}$.
\end{definition}
\begin{theorem}{\cite {aubin}}.
Let $ X$ be a paracompact differential manifold and let $ U_{i}  $ be an open cover of  $X $,  then there exists a locally finite partition of class $ C ^\infty(X)$  in  $X $, subordinate to $U_{i}$.
\end{theorem}

 \begin{definition}
 Let $(M,g)$  be an $N$-dimensional Riemannian manifold, $(\Omega _{i}, \varphi _{i} )_{i \in I} $ an atlas  of $M$  and  $(\Omega _{i}, \varphi _{i},  \eta  _{i} )$ a partition of a subordinate to $ (\Omega _{i}, \varphi _{i} ) _{i \in I}$. We can define the Riemannian measure as follows $u:M  \longrightarrow  \mathbb{R} $ with compact support by $$ \int_M u(x) \mu  _{g}(x)=\sum_{k\in J}\int_{\varphi_k(\Omega_k)}
\Big(\sqrt{\det(g_{ij})}\eta_k u\Big)\circ \varphi_k^{-1}(x)dx,$$
where $g_{ij}, i, j \in I $ are the components of the Riemannian metric $g$ in the chart  $(\Omega _{i}, \varphi _{i} )$ and $dx$ is the Lebesgue volume element of $\mathbb{R^{N}}.$\\
\end{definition}
Let $ 0<s<1 $, and  $ 1<p<\infty $ be real numbers. The fractional Sobolev space $ W^{s,p}(M) $ is defined as follows:
 It is endowed with the natural norm

$$  \|u\|_{W^{s,p}(M)}=\Bigg(\int_M \left| u(x) \right|^{p} d\mu_g(x)+[u]_{W^{s,p}(M)}^p\Bigg)^{1/p},$$ with $$ [u]_{W^{s,p}(M)}=\Bigg(\iint_{M\times M}\frac{\left| u(x)-u(y) \right|^{p}}{(d_g(x,y))^{N+ps}}d\mu
 _g(x)d\mu  _g(y)\Bigg)^{1/p},$$
where $ [u]_{W^{s,p}}$ is the Gagliardo semi-norm.\\
Let us denote by $W_{0}^{s,p}(M)$ the closure of $C_{0}^{\infty}(\Omega)$ in $W^{s,p}(M).$
Notice that $W_{0}^{s,p}(M)$ and $ W^{s,p}(M)$ are   reflexive and separable Banach spaces, for all $ 0 < s < 1 < p < \infty.$ 
We refer to \cite {r10} for more details.
\begin{lemma}  \cite {r10} \label{lemma3}
Let $(M,g)$ be an $N$-dimensional Riemannian manifold. Then,
\begin{itemize}
\item[(1)] There exists a positive constant $C_1=C_1(N,p,q,s)$ such that for any\\
 $u\in W_0^{s,p}(M)$ and $1\leq q \leq p^*_s$,
\[
\|u\|^{p}_{L^{q}(M)}\leq C_1\iint_{M\times M}\frac{\left| u(x)-u(y) \right|^{p}}{(d_g(x,y))^{N+ps}}
d\mu  _g(x)d\mu  _g(y).
\]

\item[(2)] There exists a constant $\widetilde{C}=\widetilde{C}(N,p,q,s)$ such that
 for any\\ $u\in W_0^{s,p}(M)$,
\begin{align*}
&\iint_{M\times M}\frac{\left| {u(x)-u(y)} \right|^{p}}{(d_g(x,y))^{N+ps}}d\mu  _g(x)d\mu  _g(y) 
\leq \|u\|^p_{W^{s,p}(M)} \\
&\leq \widetilde{C}\iint_{M\times M}\frac{\left| u(x)-u(y) \right|^{p}}{(d_g(x,y))^{N+ps}}
 d\mu  _g(x)d\mu  _g(y).
\end{align*}
\end{itemize}
 Consequentially, the space  $ W_0^{s,p}(M)$ is continuously embedded in $ L^{q}(M) $ for any $ q \in[p, p^*_s],$
 where $p_s^*$ is a  critical exponent defined by:
 \begin{align*}
p_s^*=
\begin{cases}
\frac{Np}{N-sp} &\text{if } sp<N,\\
\infty &\text{if } sp\geq N.
\end{cases}
\end{align*}
\end{lemma}
\begin{definition}
We say that a functional $\psi$  satisfies the Palais-Smale condition in   $ W_{0}^{s,p}(M)$, if   for any sequence $u_{n}\subset W_{0}^{s,p}(M) $ such that $\psi(u_{n}) \to c $ and $\psi'(u_{n}) \to 0$ in   $W_{0}^{s,p}(M)^{*}$ as $n\to \infty, $ then $u_{n} $ admits a convergent subsequence.
\end{definition}

 \begin{lemma} \cite {r212} \label{2-10}(Fractional Picone inequality). Let $u, v\in  W_0^{s,p}(M)$ with $ u >0$. Assume that $ (-\Delta_g)^s_p u$ is a positive bounded Radon measure in $\Omega$. Then $$  \int_{\Omega}\frac{v^{p}}{u^{p-1}} (-\Delta_g)^s_p u d\mu  _g(x) \leq  \|v\|^p_{W^{s,p}(M)}. $$
 \end{lemma}

\begin{remark}
We will use Picone's Lemma to show that the functions of type  $\displaystyle \frac{v^p}{u^p-1}$  and $\displaystyle \frac{u^p}{v^p-1}$   are  admissible test  functions, which play a crucial role in  the uniqueness part.
\end{remark}

\section{Main Results}\label{sec4}
In this section, we study the existence of  a non-trivial weak solution of problem \eqref{k1}.
\begin{definition}
A function $u\in W_0^{s,p}(\Omega)$ is said to be a weak solution of problem \eqref{k1}, if
\begin{align*}
&\iint_{M \times M}\frac {\left| u(x)-u(y) \right|^{p-2}(u(x)-u(y))(v(x)-v(y))}{(d_g(x,y))^{N+ps}}
 d\mu_g(x)d\mu_g(y)\\
 &+\int _{M} \left| u(x) \right|^{p-2} u(x)v(x)d\mu _g(x) =\int_{\Omega}f(x,u(x))v(x)d\mu  _g(x),
 \end{align*}
for any $v\in W_0^{s,p}(M)$.
\end{definition}
In the following, we will prove the existence of a non-trivial solution for the case where $\displaystyle q\in (1,p).$ 
\begin{theorem}\label{k5}
Under assumptions  {\rm (f1)}--{\rm (f3)}. If $1<q<p$, then the problem \eqref{k1}
 has a non-trivial weak solution in $W_0^{s,p}(M)$.\\
\end{theorem}
The energy functional in $ W_0^{s,p}(M)$ is defined by 
\begin{align*}
\psi(u)&=\frac{1}{p}\iint_{M\times M}\frac{\left| u(x)-u(y) \right|^{p}}{(d_g(x,y))^{N+ps}}d\mu  _g(x)
d\mu  _g(y)+   \frac{1}{p}   \int _{M} \left| u(x) \right|^{p} d\mu  _g(x)\\ 
&-\int _{\Omega} F(x, u(x)) d\mu  _g(x)\\
   &:= I_{1}(u)+ I_{2}(u)-K(u),
\end{align*}
where $\displaystyle F(x,t)=\int_0^t f(x,s)ds.$\\
The energy functional is $C ^1(W_0^{s,p}(M),\mathbb{R})$, as we will show in the following two lemmas.
 \begin{lemma}\label{lemma3-3}
Suppose that {\rm (f1)} is true,  then the functional $K\in   C ^1(W_0^{s,p}(M),\mathbb{R})$  and
\[
\langle K'(u),v\rangle=\int_\Omega f(x,u)vd\mu  _g(x), \quad
\text{for all } u,\ v\in W_0^{s,p}(\Omega).\]
\end{lemma}
\begin{proof} Let $u,v\in W_0^{s,p}(M)$,  $x\in M $ and $0<t<1,$  we have
\begin{align*}
\frac{1}{t}(F(x,u+tv)-F(x,u))&= \frac{1}{t}\int_0^{u+tv} f(m,s)ds -\frac{1}{t} \int_0^{u} f(m,s)ds\\
&= \frac{1}{t}\int_u^{u+tv} f(m,s)ds.
\end{align*}
By  the mean value theorem, there exists  a $0<z<1$ such that
$$ \frac{1}{t}(F(x,u+tv)-F(x,u)) =f(x, u+ztv)v.$$
We use the {\rm (f1)} and   Young's inequality, we obtain 
\begin{align*} 
|f(x,u+ztv)v| 
&\leq \beta |1+ |u+tzv|^{q-1} |v|\\ 
&\leq\beta |\frac{1}{q} |v|^{q}+\frac{1}{q'}+ \frac{1}{q'}|u+tvz|^{q}+ \frac{1}{q}|v|^{q}|\\
&\leq\beta (2 |v|^{q}+1+2^{2}(|v|^{q}+|u|^{q}))  \\
&\leq\beta 2^{q+1}(1+|u|^{q}+ |v|^{q}),
\end{align*} 
where $q'$ is the conjugate of $q$. Lebesgue's dominated convergence Theorem implies
\begin{align*}
\lim_{t\to 0}\frac{1}{t}(K(u+tv)-K(u))
&=\lim_{t\to 0}\int_\Omega f(x,u+z tv)vd\mu  _g(x)\\
&=\int_\Omega\lim_{t\to 0}f(x,u+z tv)vd\mu  _g(x)\\
&=\int_\Omega f(x,u)vd\mu  _g(x).
\end{align*}
Let $u{_n}, u\in W_0^{s,p}(M)$  be such that $u{_n}\to u $ strongly
in $W_0^{s,p}(M)$ as $n\to\infty$.  According to  Lemma  \ref{lemma3},  there exists a subsequence of $\{u_n\}$ still denoted by $\{u_n\}$ such that $u_n\to u$ a.e. in $\Omega$. Since $ 1<q< p^*_s $, we use the Lemma \ref{lemma3}, H\"{o}lder's inequality and {\rm (f1)},  to get  
\begin{equation} \label{m1.3}
\begin{aligned}
\int_ \Omega |f(x,u_{n})|^{q'}d\mu  _g(x)
&\leq 2^{\frac{q+1}{q-1}}\beta^{\frac{q+1}{q-1}}
 \Big(\|u_{n}\|^{q}\|_{L^{\frac{p_s^*}{q}}(\Omega)}\|1\|_{L^{\frac{p_s^*}{p_s^*-q}}(\Omega)}
 +\mu(\Omega)\Big) \\
&\leq C(\mu(\Omega))^{\frac{p_s^*-q}{p_s^*}}+C\mu(\Omega)\leq C(q, \beta,\Omega),
\end{aligned}
\end{equation}
where $\mu (\Omega)$  denotes the volume of set $\Omega$ and  $q'$ is the conjugate of $q$. It follows from inequality 4 that, the sequence $ |f(x,u_n)-f(x,u)|^{q\prime} $ is uniformly bounded and equi-integrable in $L^{1}(\Omega)$. The Vitali convergence Theorem implies
$$ \label{k8}\lim_{n\to\infty}\int_\Omega |f(x,u_n)-f(x,u)|^{q\prime}d\mu  _g(x)=0.$$
Thus, by  Lemma \ref{lemma3} and   H\"{o}lder's inequality,   we have
\begin{align*}
\|K'(u_n)-K'(u)\| _{W_0^{s,p}(M)^{*}}&= \sup_{v\in W_0^{s,p}(\Omega),\, \|v\|_{W_0^{s,p}(\Omega)}\leq1}
\|\langle K'(u_n)-K'(u),v\rangle\|\\
& \leq \|f(x,u_n)-f(x,u)\|_{L^{q'}(\Omega)}\|v\|_{L^q(\Omega)}\\
  & \leq \|f(x,u_n)-f(x,u)\|_{L^{q'}(\Omega)}\to 0,
\end{align*}
as    $ n\to\infty$, where $ W_0^{s,p}(M)^{*}$ denotes the dual space of $ W_0^{s,p}(M).$
\end{proof}
\begin{lemma}
The functional $I _{1}+ I _{2}\in C^1(W_0^{s,p}(M),\mathbb{R})$ and
\begin{align*}
\langle( I _{1}+ I _{2})'(u),v\rangle &= \int _{M} |u(x)|^{p-2} u(x)v(x)d\mu  _g(x) \\
&+ \iint_{M\times M} \frac{\left| u(x)-u(y) \right|^{p-2}(u(x)-u(y))(v(x)-v(y))}{(d_g(x,y))^{N+ps}}d\mu  _g(x)d\mu  _g(y),
\end{align*}
for all $u, v\in W_0^{s,p}(M)$.
\end{lemma}
 \begin{proof}
 Let $u, v \in  W_0^{s,p}(M)$  we have
 \begin{align*}
 \langle(  I _{2})'(u),v\rangle
 & = \lim_{t\to0}\frac{1}{t}( I _{2}(u+tv)- I _{2}(u))\\
 &=  \frac{1}{p}  \lim_{t\to 0}\frac{1}{t} \int _{M} (|u(x)+tv(x)|^{p} -|u(x)|^{p})d\mu  _g(x).
 \end{align*}
   We consider the function defined by  $K:[0,1]\to \mathbb{R}  $ as
   \[
\ K(y)=|u(x)+tyv(x)|^{p}.
\]
According to the mean value  Theorem, there exists  a $0<z<1$  such that $$\frac{1}{p} \frac{|u(x)+tv(x)|^{p} -|u(x)|^{p}}{t} = |u(x)+ztv(x)|^{p-2}(u(x)+ztv(x))v(x). $$
We  apply the mean value Theorem dominate, we get $$  \langle(  I _{2})'(u),v\rangle =  \int _{M} |u(x)|^{p-2} u(x)v(x)d\mu  _g(x).$$
  Similarly, we have  $$ \langle(  I _{1})'(u),v\rangle =\iint_{M\times M} \frac{\left| u(x)-u(y) \right|^{p-2}(u(x)-u(y))(v(x)
-v(y))}{(d_g(x,y))^{N+ps}}d\mu  _g(x)d\mu  _g(y)\cdot$$
Let $u, v \in  W_{0}^{s,p}(M)$, by H\"{o}lder's inequality, we have
\begin{align*}
 \langle( I _{1}+ I _{2})'(u),v\rangle&= \int _{M} |u(x)|^{p-2} u(x)v(x)d\mu  _g(x) \\&+\iint_{M\times M} \frac{|u(x)-u(y)|^{p-2}(u(x)-u(y))(v(x)-v(y))}{(d_g(x,y))^{N+ps}}d\mu  _g(x)d\mu  _g(y) \\
  &\leq \int _{M} |u(x)|^{p-1} v(x)d\mu  _g(x) \\
  &+ \iint_{M\times M} \frac{|u(x)-u(y)|^{p-1}(v(x)-v(y))}{(d_g(x,y))^{(N+ps)(\frac{p-1}{p})} (d_g(x,y))^{(N+ps)\frac{1}{p}}}d\mu  _g(x)d\mu  _g(y)\\
  &\leq \left( \iint_{M\times M} \frac{\left| u(x)-u(y) \right|^{p}}{(d_g(x,y))^{N+ps}}d\mu  _g(x)d\mu  _g(y)\right)^{\frac{p-1}{p}} \\&\times \left( \iint_{M\times M} \frac{\left| v(x)-v(y) \right|^{p}}{(d_g(x,y))^{N+ps}}d\mu  _g(x)d\mu  _g(y) \right)^{\frac{1}{p}} \\
  &+\left( \int _{M} |u(x)|^{p}d\mu  _g(x)\right)^{\frac{p-1}{p}} \left( \int _{M} |v(x)|^{p}d\mu  _g(x)\right)^{\frac{1}{p}}\cdot
\end{align*}
Finally, we obtain  $\psi \in  C^1(W_{0}^{s,p}(M),\mathbb{R}) $ and
\begin{align*}
 \langle( \psi )'(u),v\rangle&= \int _{M} |u(x)|^{p-2} u(x)v(x)d\mu  _g(x)\\
& + \iint_{M\times M} \frac{|u(x)-u(y)|^{p-2}(u(x)-u(y))(v(x)-v(y))}{(d_g(x,y))^{N+ps}}d\mu  _g(x)d\mu  _g(y)\cdot
\end{align*}
\end{proof}
Now, we will show that the energy functional $\psi$ is weakly lower semi-continuous, and coercive. 
\begin{lemma}

Assume {\rm (f1)} holds. Then the functional $\psi$ is weakly lower semi-continuous.
\end{lemma}
\begin{proof} Let $ \{ u_{n} \}\subset W_{0}^{s,p}(M)$, such that   $u_{n}\rightharpoonup u$ weakly in $W_0^{s,p}(M)$ as   $n\to\infty.$ Because $ I_{1}+I_{2}$ is convex functional, we concluded that the following inequality holds $$ (I_{1}+I_{2})( u_{n} )> (I_{1}+I_{2})( u )+  \langle((I_{1}+I_{2})'( u ),  u_{n}-u) \rangle.$$
Then we get that $( (I_{1}+I_{2})( u )\leq \liminf_{n\to\infty} (I_{1}+I_{2})( u _{n}).$\\
Since   $u_{n}\rightharpoonup u$ weakly in $W_{0}^{s,p}(M),$ we get that $u_{n}\to u$
strongly in $L^{q}(\Omega).$ Without loss of generality, we assume that  $u_{n}\to u$  a.e in M. Similar to the proof of  the Lemma \ref{lemma3-3}, we obtain $$\lim_{n\to\infty}\int_\Omega F(x,u_{n})d\mu  _g(x)=\int_\Omega F(x,u)d\mu  _g(x).$$ As a result,  $\psi$ is weakly lower semi-continuous in $ W_{0}^{s,p}(M).$
 \end{proof}
 Finally, by applying condition {\rm (f1)}, we can get $|F(x,z)|<2\beta(1+|z|^ {q}),$ and by applying  Lemma \ref{lemma3}, we can get 
 \begin{align*}
 \psi(u)&=\frac{1}{p}\iint_{M\times M}\frac{\left| u(x)-u(y) \right|^{p}}{(d_g(x,y))^{N+ps}}d\mu  _g(x)
d\mu  _g(y)+   \frac{1}{p}   \int _{M} \left| u(x) \right|^{p} d\mu  _g(x)\\
&- \int _{\Omega} F(x, u(x)) d\mu  _g(x) \\
&
 \geq \frac{1}{p}\iint_{M\times M}\frac{\left| u(x)-u(y) \right|^{p}}{(d_g(x,y))^{N+ps}}d\mu  _g(x)d\mu  _g(y)-\int _{\Omega} F(x, u(x)) d\mu_g(x) \\
 &\geq \frac{1}{p}\|u\|^{p}_{W_{0}^{s,p}(M)}-2\beta C_1^{\frac{q}{p}}\|u\|^{q}_{W_0^{s,p}(M)}
 -2\beta\mu(\Omega).
 \end{align*}
Since $q<p$, we have $ \psi(u)\to\infty $ as $ \|u\|_{W_{0}^{s,p}(M)}\to\infty.$  
Since $\psi$ is weakly lower semi-continuous, it has a minimum point $u_{0}$ in $ W_{0}^{s,p}(M)$, and $u_{0}$ is a weak solution of problem \eqref{k1}. This completes the proof of Theorem \ref{k5}.\\
Now, we will prove our second result, the existence of a weak solution in the case $q\in (p,p_s^*).$ We will use the geometric Mountain Pass Theorem.
\begin{theorem}\label{th2}
     Let $f$    be a function satisfying conditions {\rm (f1)}--{\rm (f4)} then the problem  \eqref{k1}  has a weak solution for $ p<q<p_s^*.$
      \end{theorem}
We must prove the following Lemmas in order to prove Theorem 3.

\begin{lemma}\label{3.7}
 The functional $\psi$ satisfies the Palais-Smale condition.
 \end{lemma}
 \begin{proof}

  Let $u_{n}   \subset W_0^{s,p}(M) $  be such that $\psi(u_{n}) \to c $ and $\psi'(u_{n}) \to 0$ in   $W_0^{s,p}(M)^{*}$ as $n\to \infty,$  so for n large  we have $c+1+ \|u\|_{W_{0}^{s,p}(M)}\geq \psi(u_{n})-\frac{1}{\mu}|\langle\psi'(u_{n}), u_{n} \rangle|$  with $\mu >0.$
 By assumption {\rm (f1)} yields 
 \begin{align*}
  c+1+ \|u_n\|_{W_{0}^{s,p}(M)} & \geq \psi(u_{n})-\frac{1}{\mu}|\langle\psi'(u_{n}), u_{n} \rangle|
  \\&= \frac{1}{p}  \|u_n\|_{W_0^{s,p}(M)}^p + \frac{1}{p}\|u_n\|_{L^{p}(M)}^p-\int _{\Omega} F(x, u_{n}(x))d\mu  _g(x)\\
  & - \frac{1}{\mu}(\|u_n\|_{W_0^{s,p}(M)}^p + \|u_n\|_{L^{p}(M)}^p)+ \frac{1}{\mu} \int _{\Omega} f(x, u_{n}(x))u_{n}(x) d\mu  _g(x)\\
   & \geq (  \frac{1}{p}- \frac{1}{\mu}) \|u_n\|_{W_{0}^{s,p}(M)}^p.
 \end{align*}
 Since $ W_0^{s,p}(M)$ is uniformly convex space,  then there is a subsequence  that will be noted  as $(u_n)$  such that $u_{n}\rightharpoonup u$ weakly in $W_{0}^{s,p}(M)$.
 We use the   H\"{o}lder's inequality, we have
 \begingroup\makeatletter\def\f@size{11.0}\check@mathfonts
\begin{align*}
 &( [u_{n}]_{s,p}^{p-1}-[u]_{s,p}^{p-1})([u_{n}]_{s,p}-[u]_{s,p})=[u_{n}]_{s,p}^{p}+[u]_{s,p}^{p}-[u_{n}]_{s,p}^{p-1}[u]_{s,p}- [u]_{s,p}^{p-1}[u_{n}]_{{s,p}}\\
&\leq \iint_{M \times M}\frac{\left| u_{n}(x)-u_{n}(y) \right|^{p}}{(d_g(x,y))^{N+ps}} d\mu_ g(x)d\mu_g(y) \\&+  \iint_{M \times M}\frac{\left| u(x)-u(y) \right|^{p}}{(d_ g(x,y))^{N+ps}} d\mu _ g(x)d\mu_g(y)\\
 &-  \iint_{M\times M} \frac{|u_n(x)-u_n(y)|^{p}(u_n(x)-u_n(y))(u(x)-u(y))}{(d_g(x,y))^{N+ps}}d\mu  _g(x)d\mu  _g(y)\\
 &- \iint_{M\times M} \frac{|u(x)-u(y)|^{p}(u(x)-u(y))(u_n(x)-u_n(y))}{(d_g(x,y))^{N+ps}}d\mu  _g(x)d\mu  _g(y)\\
& = B _{n}.
 \end{align*}
 \endgroup
  Now, we will show that $B_n \to 0$  as $n \to \infty$. We have,
  \begin{align*}
   \langle(  \psi)'(u_{n})-\psi)'(u),  u_{n}-u\rangle&= B_{n}+\int _{M} (u_{n}-u) (|u_{n}(x)|^{p-2}u_{n}- |u(x)|^{p-2}u)d\mu  _g(x)\\
   & +\int _{M} (u_{n}-u)(f(x,u_{n})-f(x,u))d\mu  _g(x).
  \end{align*}
   Since $ u_{n}\rightharpoonup u$  as $n \to \infty$  in $W_0^{s,p}(M)$.  According to  Lemma 1, we have $u_{n} \to  u $ strongly  in $L^{q}(M), $ for all $ q\in [p, p^*_s[, $ and $ u_{n} \to  u $ a.e in M. By Theorem 3.32 \cite{aubin}, there exists a sub-sequence noted by $u_{n},   h_{1}\in  L^{q}(M)$ and  $h_{2}\in  L^{q}(M)$ such that $$ |u_{n}(x)|\leq h_{1}(x),  |u(x)|\leq h_{2}(x) \text{    a.e in } M, $$ 
 for all $ q\in [p, p^*_s[. $  From   H\"{o}lder's inequality, we  get
\begingroup\makeatletter\def\f@size{10.1}\check@mathfonts
\begin{align*}
  & \int _{M}(u_{n}-u)(|u_{n}(x)|^{p-2}u_{n}-|u(x)|^{p-2}u)d\mu  _g(x)\\& \leq  \int _{M}(u_{n}-u)|u_{n}(x)|^{p-1}d\mu  _g(x)
   + \int _{M}(u_{n}-u)|u_(x)|^{p-1}d\mu  _g(x)\\
   & \leq  \int _{M}|u_{n}-u| |h_{1}|^{p-1}d\mu  _g(x)+\int _{M}|u_{n}-u||h_{2}|^{p-1}d\mu  _g(x)\\
  & \leq  \|u_{n}-u\|_{L^{p}(M)}(\| h_{1}\|_{L^{p}(M)}^{p-1}+ \|h_{2}\|_{L^{p}(M)}^{p-1}).
    \end{align*}
    \endgroup
     We also apply H\"{o}lder's inequality and the assumption {\rm (f1)} we obtained
     \begingroup\makeatletter\def\f@size{12.1}\check@mathfonts
\begin{align*}
& |\int _{M}(f(x,u_{n})- f(x,u))(u_{n}-u) d\mu  _g(x)| \\&  \leq \int _{M}|u_{n}-u| |f(x,u_{n})| d\mu  _g(x)+\int _{M}|u_{n}-u| |f(x,u)| d\mu _g(x)\\ 
 &\leq\int_{M}\beta|u_{n}-u|d\mu _g(x)+\int_{M}  |h_{1}|^{q-1} \beta|u_{n}-u|d\mu _g(x))\\
 &+\int _{M}  |h_{2}|^{q-1} \beta |u_{n}-u|d\mu _g(x)+\int _{M} \beta|u_{n}-u|d\mu  _g(x)\\
  & \leq  \beta(\|h_{1}\|_ {L^{q}(M)}^{q-1} +  \|h_{2}\|_ {L^{q}(M)}^{q-1})\|u_{n}-u\|_ {L^{q}(M)}\\
  &+  2\beta \|1\|_ {L^{\frac{p}{p-1}}(M)}\|u_{n}-u\|_ {L^{p}(M)}.
\end{align*}
\endgroup 
    
Since $u_{n}\rightharpoonup u$  as $n \to \infty,$  we have $B_{n} \to 0$ as $n\to \infty.$ Finally, we find that $ u_{n} \to u $ strongly in $ {W_0^{s,p}(M)}.$
  \end{proof}
  \begin{lemma}\label{3.8} Let   $f$   be a function satisfying  the conditions {\rm (f1)} and {\rm (f3)}. Then there are two positive real numbers, $a$ and $b,$ such that $\psi (u)\leq a$  and $ \|u\|_{W_0^{s,p}(M)}=b,$  for all $ u\in {W_0^{s,p}(M)} $ and $p<q<p_s^*.$
    \end {lemma}
 \begin{proof}
   Let's combine the two conditions {\rm (f1)} and {\rm (f3)}. There exists   a  $ C > 0 $ such that
$$ F(x,t) \leq  \varepsilon |t|^{p}+C |t|^{q},  \text{ for all } \varepsilon>0. $$
  As a result, we have       \begin{align*}
  \psi (u)&= \frac{1}{p}  \iint_{M \times M}\frac{|u(x)-u(y)|^{p}}{(d_ g(x,y))^{N+ps}} d\mu _ g(x)d\mu_g(y) +\frac{1}{p} \int_{\Omega} \frac{|u(x)|^{p}}{(d_ g(x,y))^{N+ps}} d\mu _ g(x) \\
  &-  \int_{\Omega}F(x, u(x))d\mu _ g(x)\\
  & \geq \frac{1}{p} \|u\|^p_{W_0^{s,p}(M)}- \varepsilon C_{1}(p,q)\|u\|^p_{W_0^{s,p}(M)}- C C_{1}(p,q)\|u\|^q_{W_0^{s,p}(M)}.
  \end{align*}
If we take $ \varepsilon =\frac{1}{2p C_{1}},$ we have
 \begin{align*}
 \psi (u)& \geq \frac{1}{2p} \|u\|^p_{W_0^{s,p}(M)}-C C_{1}(q,p)\|u\|^q_{W_0^{s,p}(M)}\\
 &= \|u\|^p_{W_0^{s,p}(M)}(\frac{1}{2p}-C C_{1}(q,p)\|u\|^{q-p}_{W_0^{s,p}(M)})\\
 &= b^{p} (\frac{1}{2p}-C C_{1}(q,p)b^{q-p})=:a.
 \end{align*}
\end{proof}
\begin{lemma}\label{3.9}
   If $f$   satisfying conditions {\rm (f1)} and {\rm (f4)}, if   $p<q<p_s^*,$   then there exists a  $v\in  W_0^{s,p}(M)$ such that $ \|v\|_{W_0^{s,p}(M)} \geq b $  and  $\psi (v) < 0.$  
 \end{lemma}
 \begin{proof}
    We apply the condition {\rm (f4)},  we have    $$F(x,tv)\geq t^{\mu}F(x,v), \text {for} \            t  \geq 1.$$ Consequently, we have 
 \begin{align*}
 \psi(tv)&=\frac{t^{p}}{p} \|u\|^p_{W_{0}^{s,p}(M)}+ \frac{t^{p}}{p} \|u\|^{p}_{L^{p}(M)}- \int_{\Omega }F(x,tv)d\mu _ g(x) \\
  & \leq  \frac{t^{p}}{p} (\|u\|^p_{W_0^{s,p}(M)}+C \|u\|^p_{W_0^{s,p}(M)}-\int_{\Omega}F(x,tv)d\mu _ g(x)\\
 & \leq  C' \frac{t^{p}}{p} \|u\|^p_{W_0^{s,p}(M)}-t^{\mu}\int_{\Omega }f(x,v)d\mu _ g(x).
 \end{align*}
As $\mu \geq p\geq 1, $  we have  $\psi(tv) \to -\infty  $ as $t \to \ \infty.$ So there exists $t_{0}>0$  large enough such that $\|u\|^p_{W_0^{s,p}(M)}>b$ and $\psi(t_{0}u)<0.$  We take $v = ut_{0}$ with large enough we get the result.
\end{proof}
\begin{proof} of Theorem 3.  From  lemmas (\ref{3.7})-(\ref{3.9}) and  $\psi$ satisfies  the Mountain Pass Theorem,  $\psi$ admits a critical value u; however, this u is a weak solution to the problem \eqref{k1}. As a result, the proof is finished.   
\end{proof}
\section{Uniqueness of  weak solution }\label{sec5}
Now, we   will study the following problem:

\begin{eqnarray}\label{7}
\begin{gathered}
\left\{\begin{array}{lll}
 (-\Delta_g)^s_p u(x)+|u|^{p-2} u= f(x,u) &  \text{in }& \Omega,\\
\hspace{3,2cm}  u=0 & \text{in }& M\setminus\Omega,\\
\hspace{3,2cm}  u>0 & \text{in } & \Omega.
\end{array}\right.
\end{gathered}
\end{eqnarray}
Where $ \Omega$ is an open bounded smooth-boundary subset of M.
\begin{remark}
The problem \ref{7}  is well defined.
\end{remark} 
\begin{proof}
Let u is a weak solution of our problem. Then  
\begin{eqnarray}\label{k85}
\begin{gathered}
\iint_{M \times M} \frac{|u(x)-u(y)|^{p-2}(u(x)-u(y))(v(x)-v(y))}{\left(d_{g}(x, y)\right)^{N+p s}} d \mu_{g}(x) d \mu_{g}(y) \\
+\int_{M}|u(x)|^{p-2} u(x) v(x) d \mu_{g}(x)=\int_{\Omega} f(x, u(x)) v(x) d \mu_{g}(x)
\end{gathered}
\end{eqnarray}

for any $v \in W_{0}^{s, p}(M)$. Using $v=u $ in  \ref{k85}, we have  
\begin{eqnarray}\label{k83}
\begin{gathered}
\iint_{M \times M} \frac{|u(x)-u(y)|^{p}}{\left(d_{g}(x, y)\right)^{N+p s}} d \mu_{g}(x) d \mu_{g}(y) \\
+\int_{M}|u(x)|^{p}  d \mu_{g}(x)=\int_{\Omega} f(x, u(x)) u(x) d \mu_{g}(x).
\end{gathered}
\end{eqnarray}
By condition  {\rm (f5)}, we have $u>0$. 
\end{proof}
   \begin{lemma}\label{4.2}
     Let     $v\in   W_0^{s,p}(\Omega)$  and  $g$   be a function satisfying a Lipschitz condition in $\mathbb{R}$. Then $g(v)\in W_0^{s,p}(\Omega).$
    \end{lemma}
     \begin{proof}
   Let     $v\in   W_{0}^{s,p}(\Omega).$ Using the $g$  is  a Lipschitz function we get,
   \begin{align*}
   \|g(v)\|^p_{W_0^{s,p}}(\Omega)&=\iint_{\Omega \times \Omega}\frac{|g(v(x))-g(v(y))|^{p}}{(d_g(x,y))^{N+ps}} d\mu _g(x)d\mu_g(y)\\
    &\leq l(p)\|v\|^p_{W_{0}^{s,p}}(\Omega),
   \end{align*}
   where $l$ is Lipschitz constant.
      \end{proof}
       \begin{theorem}
  Let  $f$    be a function satisfying conditions {\rm (f1)}-- {\rm (f5)},   then the problem \eqref{7} has  a unique non-trivial  solution.
 \end{theorem}
    \begin{proof}
     Let $u, v \in W_{0}^{s,p}(M) $  two solutions of problem \eqref{7},  we get

\begin{eqnarray}\label{k8}
\begin{gathered}
\iint_{M \times M}\frac{|u(x)-u(y)|^{p-2}(u(x)-u(y))(\phi(x)-\phi(y))}{(d_g(x,y))^{N+ps}}
 \,d\mu  _g(x)d\mu_g(y) \\
 +\int _{M} |u(x)|^{p-2} u(x)\phi(x)d\mu_g(x)=\int_{\Omega}f(x,u(x))\phi(x)d\mu_g(x),
\end{gathered}
\end{eqnarray}
and
\begin{eqnarray}\label{k9}
\begin{gathered}
\iint_{M \times M}\frac{|u(x)-u(y)|^{p-2}(u(x)-u(y))(\varphi(x)-\varphi(y))}{(d_g(x,y))^{N+ps}}
 d\mu_g(x)d\mu_g(y) \\+\int_{M} |u(x)|^{p-2} u(x)\psi(x)d\mu_g(x)
=\int_{\Omega}f(x,u(x))\varphi(x)d\mu_g(x),
\end{gathered}
\end{eqnarray}
for any $ \phi,  \varphi\in W_{0}^{s,p}(M)$. \\
We have $\displaystyle \frac{v^p}{u^{p-1}}\in W_0^{s,p}(M),$ thanks to Lemma \ref{4.2}.  Since  $W_0^{s,p}(M) $ is a space vector,  it yields $\displaystyle u-\frac{v^{p}}{u^{p-1}}$ and $\displaystyle \frac{u^{p}}{v^{p-1}}-v  \in W_0^{s,p}(M).$
Using  $\displaystyle \phi=u-\frac{v^{p}}{u^{p-1}}$ and $\displaystyle \varphi=\frac{u^{p}}{v^{p-1}}-v$ in \eqref{k8} and \eqref{k9} respectively, we have $$
\langle(  I _{2})'(u), \phi \rangle -  \langle(  I _{2})'(v), \varphi \rangle= \int _{\Omega}(u^{p}-v^{p}) \Big(\frac{f(x,u)}{u^{p-1}}-\frac{f(x,v)}{v^{p-1}}\Big).$$
 Thanks to  {\rm (f5)}   we obtain  \begin{equation}
\int _{\Omega}(u^{p}-v^{p}) \Big(\frac{f(x,u)}{u^{p-1}}-\frac{f(x,v)}{v^{p-1}}\Big)\leq
 0. \end{equation}
On the  other hand     \begingroup\makeatletter\def\f@size{10.1}\check@mathfonts

\begin{align*}
      & \langle(  I _{2})'(u), \phi \rangle -  \langle(  I _{2})'(v), \varphi \rangle\\&= \iint_{M \times M}\frac{|u(x)-u(y)|^{p-2}(u(x)-u(y))(\phi(x)-\phi(y))}{(d_g(x,y))^{N+ps}} \, d\mu  _g(x)d\mu_g(y)\\
       &+ \iint_{M \times M}\frac{|v(x)-v(y)|^{p-2}(v(x)-v(y))(\varphi(x)-\varphi(y))}{(d_g(x,y))^{N+ps}} \,d\mu  _g(x)d\mu_g(y)\\
       &=\|u\|^p_{W_0^{s,p}(M)}- \int _{\Omega}\frac{v^{p}}{u^{p-1}}(-\Delta_g)^s_p u+\|v\|^p_{W_0^{s,p}(M)}- \int _{\Omega}\frac{u^{p}}{v^{p-1}}(-\Delta_g)^s_p v\cdot
\end{align*}
 \endgroup
Through fractional Picone inequality, we achieve  
 \begin{equation} \langle(  I _{2})'(u), \phi \rangle -  \langle(  I _{2})'(v), \varphi \rangle\geq 0. \end{equation}
Let us collect with (4.6) and (4.7)  let us get $$\int _{\Omega}(u^{p}-v^{p}) (\frac{f(x,u)}{u^{p-1}}-\frac{f(x,v)}{v^{p-1}})=0.$$
The conesequence of  {\rm (f5)}   is  that   $u=v$ a.e in $\Omega.$
 \end{proof}
 

\begin{thebibliography}{99}
 \bibitem{bennouna}
 Aberqi, A., Bennouna, J., Benslimane, O., Ragusa, M.A.: Existence results for double phase problem in Sobolev–Orliczspaces with variable exponents in complete manifold. Mediterr. J. Math. 19, 158 (2022) 


 
 
 \bibitem{aberqi}
 Aberqi, A., Benslimane, O., Ouaziz, A., Repovs, D.D.: On a new fractional Sobolev space with variable exponent oncomplete manifolds. Bound. Value Probl. 2022, 7 (2022) 
 \bibitem{r9} A. Aberqi, J. Bennouna, M. Hammouni, 
 \emph{Non-uniformaly degenarated parabolic equations with $L^1$-data}, AIP Conference Proceedings. 2019 DOI: 10.1063/1.5090619.
\bibitem{r2} D. Applebaum,
\emph{L\'evy processes--from probability to finance quantum groups},
Notices Amer. Math. Soc, 51 (2004), 1336--1347.
\bibitem{aubin} T. Aubin,
\emph{Nonlinear analysis on manifolds},
Grundelehren der mathematischen Wissenschaften 252, A series of comprehensive Studies in Mathematics, Springer-Verlag, Berlin-Heidelberg, (1982).
\bibitem{r212} B. Barrios, I. Peral, S. Vita,
\emph{Some remarks about the summability of nonlocal nonlinear problems, } Adv. Nonlinear Anal. 4(2)(2015), 33-107.
\bibitem{benslime} O. Benslimane, A. Aberqi, J. Bennoua,
\emph{Existence and uniqueness of weak solution of $p(x)-$ laplacian    in  Sobolev spaces with variable exponents in completes manifolds}, FILOMAT. 35 (2021) 1453-1463. 
\bibitem{benslimane2020existence} O. Benslimane, A. Aberqi, J. Bennoua,
\emph{Existence and uniqueness of entropy solution of a nonlinear elliptic equation in anisotropic Sobolev--Orlicz space}, Rendiconti del Circolo Matematico di Palermo Series 2. (2020), 1-30. https://doi.org/10.1007/s12215--020--00577--4.
\bibitem{bile}
P. Biler, G. Karch, W. A. Woyczy\'nski, 
\emph{Critical nonlinearity exponent and self-similar asymptotics for Lévy conser-
vation laws,} Ann. Inst. H. Poincaré Anal. Non Linéaire 18 (5) (2001) 613–-637.
\bibitem{Radulescu} G. M. Bisci, V. D. Radulescu, R. Servadei,
\emph{Variational methods for nonlocal fractional problems},
volume 162 of Encyclopedia of Mathematics and Applications. Cambridge University Press, Cambridge (2016).
\bibitem{caffarelli}
L. Caffarelli, J.-M. Roquejoffre, O. Savin, \emph{Non-local minimal surfaces,} Comm. Pure Appl. Math. 63 (2010) 1111–-1144.
\bibitem{valdinici}
L. Caffarelli, E. Valdinoci, \emph{Uniform estimates and limiting arguments for nonlocal minimal surfaces, } Calc. Var. Partial Differential Equations 41 (1–2) (2011) 203–-240.
\bibitem{CG} S. Y. A. Chang, M. d. M. Gonz\'{a}lez,
\emph{Fractional Laplacian in conformal geometry}, Adv. Math, 226 (2010), 1410--1432.
 \bibitem{conte}
 R. Cont, P. Tankov,
  \emph{Financial Modelling with Jump Processes, } Chapman  Hall/CRC Financ. Math. Ser., Chapman  Hall/CRC, Boca Raton, FL, 2004.
\bibitem{A-ph} D. Danielli, S. Salsa,
\emph{Obstacle problems involving the fractional Laplacian},
In Recent developments in nonlocal theory, pages 81--164. De Gruyter, Berlin, (2018).
\bibitem{druet}
O. Druet, E. Hebey, 
\emph{Blow-up examples for second order elliptic PDEs of critical Sobolev
growth,} Trans. Amer. Math. Soc. 357 (2004) 1915–-1929.
\bibitem{r8} E. Di Nezza, G. Palatucci, E. Valdinoci;
\emph{Hitchhiker's guide to the fractional Sobolev spaces},
Bull. Sci. Math, 136 (2012), 521--573.
\bibitem{dislocation}  S. Dipierro, A. Figalli,  E. Valdinoci, 
\emph{Strongly nonlocal dislocation dynamics in crystals},
Comm. Partial Differential Equations, 39(12)(2014), 2351--2387.
\bibitem{duvaut}
G. Duvaut, J. L. Lions, 
\emph{Inequalities in Mechanics and Physics,} Grundlehren Math. Wiss., vol. 219, Springer-Verlag,
Berlin, 1976. Translated from French by C.W. John.
\bibitem{FSV} A. Fiscella, R. Servadei, E. Valdinoci,
\emph{Density properties for fractional Sobolev spaces},
Ann. Acad. Sci. Fenn. Math, 40 (2015), 235--253.
\bibitem{FSV1}  P. M. Gadea, J. M. Masqué,  I. V. Mykytyuk,
\emph{Analysis and algebra on differentiable manifolds:
a workbook for students and teachers}, Springer Science \& Business Media, (2012).
\bibitem{MG} M. d. M. Gonz\'alez,
\emph{Recent progress on the fractional Laplacian in conformal geometry}, arXiv:1609.08988.
\bibitem{GQ} M. d. M. Gonz\'alez, J. Qing,
\emph{Fractional conformal Laplacians and fractional Yamabe
problems}, Analysis \& PDE, 6.7 (2013), 1535--1576.
\bibitem{r10} L. Guo,
\emph{Fractional p-Laplacian equations on Riemannian manifolds}, Electronic Journal of Differential Equations, 156(2018), 1-–17.
\bibitem{guo2} E. Hebey,
\emph{Nonliner analysis on manifolds: Sobolev spaces and inequalities,}
Courant Lecture Notes, vol. 5, American Mathematical Society, (2000).
\bibitem{r12} A. Iannizzotto, M. Squassina,
\emph{Weyl-type laws for fractional $p$-eigenvalue problems},
Asymptotic Anal, 88 (2014), 233--245.
\bibitem{ilyas}
S. Ilias, 
\emph{Constantes explicites pour les in´egalit´es de Sobolev sur les vari´et´es riemanniennes
compactes,} Ann. Inst. Fourier 33 (1983) 151-–165.
\bibitem{M} G. Molica Bisci,
\emph{Sequences of weak solutions for fractional equations}, Math. Res. Lett, 21 (2014), 1--13.
\bibitem{r-4} G. Molica Bisci,
\emph{Fractional equations with bounded primitive}, Appl. Math. Lett., 27 (2014).
\bibitem{savin}
O. Savin, E. Valdinoci, 
\emph{Elliptic PDEs with fibered nonlinearities, } J. Geom. Anal. 19 (2) (2009) 420–-432.
\bibitem{r17} R. Servadei, E. Valdinoci,
 \emph{Lewy-Stampacchia type estimates for variational inequalities driven by (non)local operators, } Rev. Mat. Iberoam, 29 (2013), 1091--1126.
 \bibitem{trudinger}
N. S. Trudinger, 
\emph{Remarks concerning the conformal deformation of Riemannian structures
on compact manifolds,} Ann. Scuola Norm. Sup. Pisa 22 (1968) 265–-274.
\end{thebibliography}
\end{document}